\documentclass[12pt]{amsart}
\usepackage{amssymb}
\usepackage{enumerate}
\usepackage{graphicx}
\usepackage[all,cmtip]{xy}
 \usepackage{tikz-cd}

\makeatletter
\@namedef{subjclassname@2010}{%
  \textup{2010} Mathematics Subject Classification}
\makeatother
\newtheorem{thm}{Theorem}[section]

\newtheorem{corollary}[thm]{Corollary}
\newtheorem{lemma}[thm]{Lemma}
\newtheorem{proposition}[thm]{Proposition}
\theoremstyle{definition}
\newtheorem{definition}[thm]{Definition}
\newtheorem{remark}[thm]{Remark}
\newtheorem{example}[thm]{Example}
\begin{document}
\baselineskip=15pt
\title{Uniformly $S$-projective relative to a module and its dual}

\author[M. Adarbeh]{Mohammad Adarbeh $^{(\star)}$}
\address{Department of Mathematics, Birzeit University, Birzeit,  Palestine}
\email{madarbeh@birzeit.edu}
\author[M. Saleh]{Mohammad Saleh }
\address{Department of Mathematics, Birzeit University, Birzeit,  Palestine}
\email{msaleh@birzeit.edu}

\thanks{$^{(\star)}$ Corresponding author}
\date{}

\begin{abstract}
In this article, we introduce the notion of uniformly $S$-projective ($u$-$S$-projective) relative to a module. Let $S$ be a multiplicative subset of a ring $R$ and $M$ an $R$-module. An $R$-module $P$ is said to be $u$-$S$-projective relative to $M$ if for any $u$-$S$-epimorphism $f:M\to N$, the induced map 
 $\text{Hom}_{R}(P,f): \text{Hom}_R(P,M)\to \text{Hom}_R(P,N)$ is a $u$-$S$-epimorphism. Dually, we also introduce $u$-$S$-injective relative to a module. Some properties of these notions are discussed. Several characterizations of $u$-$S$-semisimple modules are given in terms of these notions. The notions of $u$-$S$-quasi-projective and $u$-$S$-quasi-injective modules are also introduced, and some of their properties are discussed. 
\end{abstract}



\maketitle

\section{Introduction}
In this paper, all rings are commutative with a nonzero identity, and all modules are unitary. A subset $S$ of a ring $R$ is said to be a multiplicative subset of $R$ if $1 \in S$, $0 \notin S$, and $s_1s_2\in S$ for all $s_1,s_2 \in S$. Throughout, $R$ denotes a commutative ring and $S$ a multiplicative subset of $R$. Recall that an $R$-module $M$ is called a $u$-$S$-torsion module if there exists $s \in S$ such that $sM = 0$ \cite[Definition 2.1]{Z}. Let $M, N, L$ be $R$-modules.
    \begin{enumerate}
    \item[(i)] An $R$-homomorphism $f: M \to N$ is called a $u$-$S$-monomorphism ($u$-$S$-epimorphism) if $\text{Ker}(f)$ ($\text{Coker}(f)$) is a $u$-$S$-torsion module \cite{Z}.
    \item[(ii)] An $R$-homomorphism $f: M \to N$ is called a $u$-$S$-isomorphism if $f$ is both a
 $u$-$S$-monomorphism and a $u$-$S$-epimorphism \cite{Z}.
\item[(iii)]  An $R$-sequence $M \xrightarrow{f} N \xrightarrow{g} L$ is said to be $u$-$S$-exact if there exists $s \in S$ such that $s\text{Ker}(g) \subseteq \text{Im}(f)$ and $s\text{Im}(f) \subseteq \text{Ker}(g)$. A $u$-$S$-exact sequence $0 \to M \to N \to L \to 0$ is called a short $u$-$S$-exact sequence \cite{ZQ}. 

\item[(iv)] A short $u$-$S$-exact sequence $0\to M \xrightarrow{f} N \xrightarrow{g} L\to 0$ is said to be $u$-$S$-split (with respect to $s$) if there is $s\in S$ and an $R$-homomorphism $f':N \to M$ such that $f'f=s1_{M}$, where $1_{M}:M\to M$ is the identity map on $M$ \cite{ZQ}.
 \end{enumerate}
  W. Qi and H. Kim et al. \cite{QK} introduced the notion of $u$-$S$-injective modules. They defined an $R$-module $E$ to be $u$-$S$-injective if the induced sequence
 \begin{center}
     $0 \to \text{Hom}_{R}(C,E)\to \text{Hom}_R(B,E)\to \text{Hom}_R(A,E) \to 0$
 \end{center}
 is $u$-$S$-exact for any $u$-$S$-exact sequence $0 \to
A\to B\to C \to 0$. X. L. Zhang and W. Qi \cite{ZQ} introduced the notion of $u$-$S$-projective modules as a dual notion of $u$-$S$-injective modules. They defined an $R$-module $P$ to be $u$-$S$-projective if the induced sequence
 \begin{center}
$0 \to \text{Hom}_{R}(P,A)\to \text{Hom}_R(P,B)\to \text{Hom}_R(P,C) \to 0$     
 \end{center}
 is $u$-$S$-exact for any $u$-$S$-exact sequence $0 \to
A\to B\to C \to 0$. They also introduced the notions of $u$-$S$-semisimple modules and $u$-$S$-semisimple rings. An $R$-module $M$ is said to be $u$-$S$-semisimple if any $u$-$S$-short exact sequence $0\to A\to M\to C\to 0$ is $u$-$S$-split. A ring $R$ is said to be $u$-$S$-semisimple if any free $R$-module is $u$-$S$-semisimple. Recall that an $R$-module $U$ is said to be projective relative to a module $M$ if for each epimorphism $f:M\to N$, the induced map $$\text{Hom}_{R}(U,f):\text{Hom}_R(U,M)\to \text{Hom}_R(U,N)$$ is an epimorphism \cite{AF}. Dually, an $R$-module $U$ is said to be injective relative to a module $M$ if for each monomorphism $g:K\to M$, the induced map $$\text{Hom}_{R}(g,U): \text{Hom}_R(M,U)\to \text{Hom}_R(K,U)$$ is an epimorphism \cite{AF}. An $R$-module $M$ is said to be quasi-projective (quasi-injective) if $M$ is  projective (injective) relative to $M$ \cite{AF}. The ''uniformly'' $S$-versions of projective (injective) relative to a module as well as the ''uniformly'' $S$-versions of quasi-projective (quasi-injective) modules are given in this article (see Definitions \ref{def1} and \ref{def2}). The notions of $u$-$S$-quasi-injective and $u$-$S$-quasi-projective modules have been generalized by M. Adarbeh and M. Saleh in \cite{MM1} and \cite{MM2}, respectively.

In Section 2, we first introduce the notions of $u$-$S$-projective relative to a module and $u$-$S$-injective relative to a module, and then we investigate some properties of these notions. For example, in Theorems \ref{thm1} and \ref{thm2}, we show that an $R$-module $U$ is $u$-$S$-projective ($u$-$S$-injective) relative to a module $M$ if and only if for any submodule $K$ of $M$, $\text{Hom}_R(U,\eta_K)$ ($\text{Hom}_R(i_K,U)$) is a $u$-$S$-epimorphism, where $\eta_K:M\to \frac{M}{K}$ is the natural map and $i_K:K\to M$ is the inclusion map. We prove in Theorems \ref{thm6} and \ref{thrm} that an $R$-module $M$ is $u$-$S$-semisimple if and only if every $R$-module is $u$-$S$-injective relative to $M$ if and only if every $R$-module is $u$-$S$-projective relative to $M$ if and only if for every injective $R$-module $U$ and $f\in \text{End}_{R}(U)$, $\text{Ker}(f)$ is $u$-$S$-injective relative to $M$ if and only if for every projective $R$-module $U$ and $f\in \text{End}_{R}(U)$, $\text{Coker}(f)$ is $u$-$S$-projective relative to $M$. In Proposition \ref{localr}, we give a new characterization of projective (injective) relative to a module.

In Section 3, we introduce $u$-$S$-quasi-projective and $u$-$S$-quasi-injective modules, and then we discuss some properties of these notions. By Remark \ref{rem2} (1), we have the following implications:
\begin{center}
    $u$-$S$-projective $\Rightarrow$ $u$-$S$-quasi-projective\end{center} and
\begin{center}
 $u$-$S$-injective $\Rightarrow$ $u$-$S$-quasi-injective.
 \end{center} 
The above implications are irreversible in general (see Examples \ref{ex1} and \ref{ex2}). In Proposition \ref{local}, we give a local characterization of quasi-projective (quasi-injective) modules. Finally, in Theorem \ref{thm7}, we give a characterization of $u$-$S$-semisimple rings in terms of $u$-$S$-quasi-injective and $u$-$S$-quasi-projective modules.

\section{$u$-$S$-projective ($u$-$S$-injective) relative to a module }\label{d}

We start this section with the following definition:

\begin{definition}\label{def1}
   Let $S$ be a multiplicative subset of a ring $R$ and $M$ an $R$-module.
\begin{enumerate}
       \item[(i)] An $R$-module $P$ is said to be $u$-$S$-projective relative to $M$ if for any $u$-$S$-epimorphism $f:M\to N$, the induced map 
 $$\text{Hom}_{R}(P,f): \text{Hom}_R(P,M)\to \text{Hom}_R(P,N)$$ is a $u$-$S$-epimorphism.
       \item[(ii)] An $R$-module $E$ is said to be $u$-$S$-injective relative to $M$ if for any $u$-$S$-monomorphism $f:K\to M$, the induced map 
 $$\text{Hom}_{R}(f,E): \text{Hom}_R(M,E)\to \text{Hom}_R(K,E)$$ is a $u$-$S$-epimorphism.
\end{enumerate}   
\end{definition}

 \begin{remark}\label{rem1}
Let $S$ be a multiplicative subset of a ring $R$ and let $P,E$ be $R$-modules. 
\begin{enumerate}
    \item[(a)] By \cite[Proposition 2.9]{ZQ}, $P$ is $u$-$S$-projective if and only if $P$ is $u$-$S$-projective relative to every $R$-module $M$.
    \item[(b)] By \cite[Proposition 2.5]{ZQ}, $E$ is $u$-$S$-injective if and only if $E$ is $u$-$S$-injective relative to every $R$-module $M$.
    \end{enumerate}
\end{remark}

For an $R$-module $M$, an $R$-homomorphism $f:A\to B$, and a submodule $K$ of $M$ ($K\leq M$), let $f_{*}$ ($f^{*}$) denote the map $\text{Hom}_{R}(M,f)$ ($\text{Hom}_{R}(f,M)$), $i_K:K\to M$ denote the inclusion map, and $\eta_K:M\to \frac{M}{K}$ denote the natural map. 

\begin{thm}\label{thm1}
  Let $S$ be a multiplicative subset of a ring $R$ and let $M,P$ be $R$-modules. Then the following statements are equivalent:
\begin{enumerate}
    \item[(1)] $P$ is $u$-$S$-projective relative to $M$;
    \item[(2)] for any epimorphism $g:M\to N$, the map 
 $$g_*: \text{Hom}_R(P,M)\to \text{Hom}_R(P,N)$$ is a $u$-$S$-epimorphism;
    \item[(3)] for any $K\leq M$, $(\eta_K)_{*}$ is a $u$-$S$-epimorphism.
\end{enumerate}
\end{thm}

\begin{proof}
    $(1)\Rightarrow (2)$: Clear.\\
$(2)\Rightarrow (3)$: Apply $(2)$ to the epimorphism $\eta_{K}:M\to \frac{M}{K}$.\\
    $(3)\Rightarrow (1)$: Let $f:M\to N$ be a $u$-$S$-epimorphism and $K=\text{Ker}(f)$. Then $g:\frac{M}{K}\to N$ given by $g(x+K)=f(x)$, $x\in M$, is a $u$-$S$-isomorphism and $f=g\eta_K$, where $\eta_{K}:M\to \frac{M}{K}$ is the natural map. By \cite[Lemma 2.1]{ZQ}, there is a $u$-$S$-isomorphism $h:N\to \frac{M}{K}$ and $t\in S$ such that $hg =t1_\frac{M}{K}$. So $hf=hg\eta_K=t\eta_K$. Since $h$ is a $u$-$S$-monomorphism, $t'\text{Ker}(h)=0$ for some $t'\in S$. Now by $(3)$, $(\eta_K)_{*}:\text{Hom}_R(P,M)\to \text{Hom}_R(P,\frac{M}{K})$ is a $u$-$S$-epimorphism. So there is $s\in S$ such that $s\text{Hom}_R(P,\frac{M}{K})\subseteq \text{Im}(\eta_K)_*$. Take $\alpha \in \text{Hom}_R(P,N)$. Then $h\alpha\in \text{Hom}_R(P,\frac{M}{K})$. So $sh\alpha= (\eta_K)_{*}(\beta)=\eta_K\beta$ for some $\beta \in \text{Hom}_R(P,M)$. Hence $tsh\alpha = t\eta_K\beta=hf\beta$. It is easy to check $t'ts\alpha=t'f\beta=ft'\beta$. So $t'ts\alpha\in \text{Im}(f_*)$. Thus $f_*$ is a $u$-$S$-epimorphism and therefore, $(1)$ holds.   
    \end{proof}

The following result is the dual of Theorem \ref{thm1}.

\begin{thm}\label{thm2}
  Let $S$ be a multiplicative subset of a ring $R$ and let $E,M$ be $R$-modules. Then the following statements are equivalent:
\begin{enumerate}
    \item[(1)] $E$ is $u$-$S$-injective relative to $M$; 
    \item[(2)] for any monomorphism $f:K\to M$, the map 
 $$f^{*}: \text{Hom}_R(M,E)\to \text{Hom}_R(K,E)$$ is a $u$-$S$-epimorphism; 
    \item[(3)] for any $K\leq M$, $(i_K)^{*}$ is a $u$-$S$-epimorphism.
\end{enumerate}
\end{thm}

\begin{corollary}\label{coro}
    Let $S$ be a multiplicative subset of a ring $R$ and let $M,U$ be $R$-modules. If $U$ is projective (injective) relative to $M$, then $U$ is $u$-$S$-projective ($u$-$S$-injective) relative to $M$.
\end{corollary}

\begin{lemma}\label{lemm1}
     Let $S$ be a multiplicative subset of a ring $R$, $f_i:A_i\to B_i$, $i=1,\cdots,n$ be $R$-homomorphisms, and $\bigoplus\limits_{i=1}^{n}f_i:\bigoplus\limits_{i=1}^{n}A_i\to \bigoplus\limits_{i=1}^{n}B_i$ be the direct sum map of $(f_i)_{i=1}^{n}$. Then $\bigoplus\limits_{i=1}^{n}f_i$ is a $u$-$S$-epimorphism if and only if each $f_i$ is a $u$-$S$-epimorphism.
\end{lemma}

\begin{proof}
  Let $f:=\bigoplus\limits_{i=1}^{n}f_i$. Suppose that $f$ is a $u$-$S$-epimorphism, then $s\bigoplus\limits_{i=1}^{n}B_i\subseteq \text{Im}(f)=\bigoplus\limits_{i=1}^{n}\text{Im}(f_i)$ for some $s\in S$. So $sB_i\subseteq \text{Im}(f_i)$ for each $i=1,\cdots,n$. Hence, each $f_i$ is a $u$-$S$-epimorphism. Conversely, suppose that each $f_i$ is a $u$-$S$-epimorphism. Then for each $i=1,\cdots,n$, there is $s_i\in S$ such that $s_iB_i\subseteq \text{Im}(f_i)$. Let $s=s_1\cdots s_n$. Then $sB_i\subseteq \text{Im}(f_i)$ for each $i=1,\cdots,n$. Hence $s\bigoplus\limits_{i=1}^{n}B_i=\bigoplus\limits_{i=1}^{n}sB_i\subseteq \bigoplus\limits_{i=1}^{n}\text{Im}(f_i)=\text{Im}(f)$. Thus $f$ is a $u$-$S$-epimorphism.
\end{proof}

\begin{thm}\label{thm3}
  Let $S$ be a multiplicative subset of a ring $R$ and $M,A_1,\cdots,A_n$ be $R$-modules. Then the following statements hold.
\begin{enumerate}
       \item[(1)] $\bigoplus\limits_{i=1}^{n}A_i$ is $u$-$S$-projective relative to $M$ if and only if each $A_i$ is $u$-$S$-projective relative to $M$.
       \item[(2)] $\bigoplus\limits_{i=1}^{n}A_i$ is $u$-$S$-injective relative to $M$ if and only if each $A_i$ is $u$-$S$-injective relative to $M$.
       \end{enumerate}
\end{thm}

\begin{proof}
We will prove only part (1), as the other part is a dual of it. Let $f:M\to N$ be a $u$-$S$-epimorphism. Then there are natural isomorphisms $\alpha$ and $\beta$ such that the following diagram 
  \[\xymatrix{
 \text{Hom}_{R}\big(\bigoplus\limits_{i=1}^{n}A_i,M\big)\ar[d]_{\alpha}\ar[r]^{ \theta } & \text{Hom}_{R}\big(\bigoplus\limits_{i=1}^{n}A_i,N\big)\ar[d]_{\beta}\\
\bigoplus\limits_{i=1}^{n}\text{Hom}_{R}\big(A_i,M\big)\ar[r]^{\lambda }  & \bigoplus\limits_{i=1}^{n}\text{Hom}_{R}\big(A_i,N\big)
}\] commutes, where $\theta=\text{Hom}_{R}\big(\bigoplus\limits_{i=1}^{n}A_i,f\big)$ and $\lambda=\bigoplus\limits_{i=1}^{n}\text{Hom}_{R}(A_i,f)$ \cite{AF}. Hence $\text{Hom}_{R}\big(\bigoplus\limits_{i=1}^{n}A_i,f\big)$ is a $u$-$S$-epimorphism if and only if $\bigoplus\limits_{i=1}^{n}\text{Hom}_{R}(A_i,f)$ is a $u$-$S$-epimorphism if and only if each $\text{Hom}_{R}(A_i,f)$ is a $u$-$S$-epimorphism  by Lemma \ref{lemm1}. Thus $(1)$ holds.  
\end{proof}

\begin{corollary}
      Let $S$ be a multiplicative subset of a ring $R$ and $A_1,\cdots,A_n$ be $R$-modules. Then the following statements hold.
      \begin{enumerate}
       \item[(1)] $\bigoplus\limits_{i=1}^{n}A_i$ is $u$-$S$-projective if and only if each $A_i$ is $u$-$S$-projective.
       \item[(2)] $\bigoplus\limits_{i=1}^{n}A_i$ is $u$-$S$-injective if and only if each $A_i$ is $u$-$S$-injective.
       \end{enumerate}
\end{corollary}

\begin{proposition}\label{pr1}
     Let $S$ be a multiplicative subset of a ring $R$ and $M$ an $R$-module. Let $f:A\to B$ be a $u$-$S$-isomorphism. Then the following statements hold.
     \begin{enumerate}
         \item[(1)] $A$ is $u$-$S$-projective ($u$-$S$-injective) relative to $M$ if and only if $B$ is $u$-$S$-projective ($u$-$S$-injective) relative to $M$. 
         \item[(2)] $M$ is $u$-$S$-projective ($u$-$S$-injective) relative to $A$ if and only if $M$ is $u$-$S$-projective ($u$-$S$-injective) relative to $B$.
     \end{enumerate}
\end{proposition}

\begin{proof}
We will prove only the case of relative $u$-$S$-projectivity. First, since  $f:A\to B$ is a $u$-$S$-isomorphism, then by \cite[Lemma 2.1]{ZQ}, there is a $u$-$S$-isomorphism $f':B\to A$ and $t\in S$ such that $ff'=t1_B$ and $f'f=t1_A$. \\[0.1cm]
(1) Let $g:M\to N$ be a $u$-$S$-epimorphism. Since $A$ is $u$-$S$-projective relative to $M$, then the map $$\text{Hom}_{R}(A,g): \text{Hom}_R(A,M)\to \text{Hom}_R(A,N)$$ is a $u$-$S$-epimorphism. So $s\text{Hom}_R(A,N)\subseteq \text{Im}\big(\text{Hom}_{R}(A,g)\big)$ for some $s\in S$. Take $h\in \text{Hom}_{R}(B,N)$. Then $hf\in \text{Hom}_{R}(A,N)$, so $shf=gg'$ for some $g'\in \text{Hom}_R(A,M)$. So $sth=sth1_B=shff'=gg'f'\in \text{Im}\big(\text{Hom}_{R}(B,g)\big)$ (since $g'f'\in \text{Hom}_R(B,M)$). Hence, the map $$\text{Hom}_{R}(B,g): \text{Hom}_R(B,M)\to \text{Hom}_R(B,N)$$ is a $u$-$S$-epimorphism. Thus $B$ is $u$-$S$-projective relative to $M$. The other direction is similar. \\[0.1cm]
(2) Let $g:B\to C$ be a $u$-$S$-epimorphism. Since $f:A\to B$ is a $u$-$S$-epimorphism, $gf:A\to C$ is a $u$-$S$-epimorphism. But $M$ is $u$-$S$-projective relative to $A$, so $(gf)_*=\text{Hom}_{R}(M,gf)=\text{Hom}_{R}(M,g)\text{Hom}_{R}(M,f)=g_*f_*$ is a $u$-$S$-epimorphism. So $s\text{Hom}_{R}(M,C)\subseteq \text{Im}(g_*f_*)\subseteq \text{Im}(g_*)$ for some $s\in S$. Thus $g_*$ is a $u$-$S$-epimorphism. Therefore, $M$ is $u$-$S$-projective relative to $B$. For the converse, if $h:A\to C$ is a $u$-$S$-epimorphism, then $hf':B\to C$ is a $u$-$S$-epimorphism. So $(hf')_*=\text{Hom}_{R}(M,hf')=\text{Hom}_{R}(M,h)\text{Hom}_{R}(M,f')=h_*f'_*$ is a $u$-$S$-epimorphism and hence $h_*$ is a $u$-$S$-epimorphism. Thus $M$ is $u$-$S$-projective relative to $A$.
\end{proof}

\begin{lemma}\label{lem1}
    Let $S$ be a multiplicative subset of a ring $R$. Consider the following commutative diagram with exact rows:
   \[\xymatrix{
A\ar[d]_{\alpha} \ar[r]^{f} & B\ar[d]_{\beta} \ar[r]^{g} & C\ar[d]_{\gamma}\\
D\ar[r]^{h}  & E \ar[r]^{k}& F
}\] 

\begin{enumerate}
    \item[(1)] If $\beta$ is a $u$-$S$-epimorphism and $h,\gamma$ are monomorphisms, then $\alpha$ is a $u$-$S$-epimorphism.
    \item[(2)] If $\alpha,\gamma,g$ are $u$-$S$-epimorphisms, then $\beta$ is a $u$-$S$-epimorphism.
    \item[(3)] If $0\to A\xrightarrow{f} B\xrightarrow{g} C\to 0$ is a $u$-$S$-split exact sequence, then $\text{Hom}_R(M,g)$ is a $u$-$S$-epimorphism for any $R$-module $M$. 
\end{enumerate}
\end{lemma}

\begin{proof}
(1) Since $\beta$ is a $u$-$S$-epimorphism, so $sE\subseteq \text{Im}(\beta)$ for some $s\in S$. Let $d\in D$. Then $h(d)\in E$ implies $sh(d)=\beta(b)$ for some $b\in B$. So $k\beta (b)=skh(d)=0$ but then $\gamma g(b)=0$. But $\gamma$ is a monomorphism, so $g(b)=0$ and this implies $b\in \text{Ker}(g)=\text{Im}(f)$. So $b=f(a)$ for some $a\in A$. Hence $sh(d)=\beta(b)=\beta (f(a))=\beta f(a)=h\alpha(a)$. Since $h$ is a monomorphism, $sd=\alpha(a)$. Hence $sD\subseteq \text{Im}(\alpha)$. Thus $\alpha$ is a $u$-$S$-epimorphism. \\[0.1cm]
(2) Similar to the proof of \cite[Lemma 2.11]{QK}.\\[0.1cm]
(3) By \cite[Lemma 2.4]{ZQ}, there is $s\in S$ and $g':C\to B$ such that $gg'=s1_C$. For any $h\in \text{Hom}_R(M,C)$, we have $sh=s1_Ch=gg'h\in \text{Im}\big(\text{Hom}_R(M,g)\big)$. Thus $\text{Hom}_R(M,g)$ is a $u$-$S$-epimorphism. 
\end{proof}

\begin{thm}\label{thm4} 
     Let $S$ be a multiplicative subset of a ring $R$ and $M$ an $R$-module. Then the following statements hold.
    \begin{enumerate}
        \item[(1)] Let $A$ be a submodule of $B$. If $M$ is $u$-$S$-projective relative to $B$, then $M$ is $u$-$S$-projective relative to both $A$ and $\frac{B}{A}$.
        \item[(2)] $M$ is $u$-$S$-projective relative to $\bigoplus\limits_{i=1}^{n}A_i$ if and only if $M$ is $u$-$S$-projective relative to each $A_i$.
    \end{enumerate}
\end{thm}

\begin{proof} (1) Since $\eta_{A}:B\to \frac{B}{A}$ is an epimorphism and $M$ is $u$-$S$-projective relative to $B$, then by the first part of the proof of Proposition \ref{pr1} (2), we have $M$ is $u$-$S$-projective relative to $\frac{B}{A}$. Now, let $K\leq A$ and consider the commutative diagram 

     \[\xymatrix{
      & & & 0\ar[d]\\
0\ar[r] & A\ar[d]_{\alpha}\ar[r] & B\ar[d]_{\beta} \ar[r]&\frac{B}{A}\ar[d]_{\gamma}\ar[r]&0\\
0\ar[r]&\frac{A}{K}\ar[d]\ar[r]^{\sigma}  & \frac{B}{K}\ar[d] \ar[r]&\frac{B}{A}\ar[d]\ar[r]&0\\
& 0  &0  &0
}\] 

 with exact rows and columns, where $\alpha,\beta$ are the natural maps and $\gamma=1_{\frac{B}{A}}$ is the identity map. Since $\text{Hom}_{R}(M,-)$ is left exact, we have the following commutative diagram 

    \[\xymatrix{
      & & & \\
0\ar[r] & \text{Hom}_{R}(M,A)\ar[d]_{\alpha_*}\ar[r] & \text{Hom}_{R}(M,B)\ar[d]_{\beta_*} \ar[r]&\text{Hom}_{R}(M,\frac{B}{A})\ar[d]_{\gamma_*}\\
0\ar[r]&\text{Hom}_{R}(M,\frac{A}{K})\ar[r]^{\sigma_*}  & \text{Hom}_{R}(M,\frac{B}{K}) \ar[r]& \text{Hom}_{R}(M,\frac{B}{A})\\
 & &  &
}\] 
 with exact rows. Since $M$ is $u$-$S$-projective relative to $B$, $\beta_*$ is a $u$-$S$-epimorphism. But $\sigma_*$ and $\gamma_*$ are monomorphisms, so by Lemma \ref{lem1} (1), $\alpha_*$ is a $u$-$S$-epimorphism. Thus, by Theorem \ref{thm1}, $M$ is $u$-$S$-projective relative to $A$. \\[0.1cm]
(2) We will prove this part only for $n=2$. Suppose that $M$ is $u$-$S$-projective relative to $A \oplus B$. Since $\frac{A\oplus B}{A}\cong B$, then by part (1) and Proposition \ref{pr1} (2), $M$ is $u$-$S$-projective relative to both $A$ and $B$. Conversely, suppose that $M$ is $u$-$S$-projective relative to both $A$ and $B$. Let $K\leq A\oplus B$. Then there is an obvious commutative diagram 
           \[\xymatrix{
     0\ar[r] & A\ar[d]\ar[r] & A\oplus B\ar[d]_{\eta_K} \ar[r]^{\pi}&B\ar[d]\ar[r]&0\\
0\ar[r]&\frac{A+K}{K}\ar[d]\ar[r]  & \frac{A\oplus B}{K}\ar[d] \ar[r]&\frac{A\oplus B}{A+K}\ar[d]\ar[r]&0\\
& 0  &0  &0
}\] with exact rows and columns \cite{AF}, where $\pi$ is the natural projection. Applying $\text{Hom}_{R}(M,-)$ to the above diagram, we obtain the following commutative diagram 

\[\xymatrix{
     0\ar[r] & \text{Hom}_{R}(M,A)\ar[d]\ar[r] & \text{Hom}_{R}(M,A\oplus B)\ar[d]_{(\eta_K)_*} \ar[r]^{\pi_*}&\text{Hom}_{R}(M,B)\ar[d]\\
0\ar[r]&\text{Hom}_{R}(M,\frac{A+K}{K})\ar[d]\ar[r]  & \text{Hom}_{R}(M,\frac{A\oplus B}{K}) \ar[r]&\text{Hom}_{R}(M,\frac{A\oplus B}{A+K})\ar[d]\\
& 0  &  &0
}\] 
with exact rows. Since $M$ is $u$-$S$-projective relative to both $A$ and $B$, then the first and the third columns of the last diagram are $u$-$S$-exact. But the sequence $0 \to A\to A\oplus B\to B\to 0$ splits, so $\pi_*$ is a $u$-$S$-epimorphism by Lemma \ref{lem1} (3). Hence $(\eta_K)_*$ is a $u$-$S$-epimorphism by Lemma \ref{lem1} (2). Therefore, by Theorem \ref{thm1}, $M$ is $u$-$S$-projective relative to $A\oplus B$.
\end{proof}

\begin{thm}\label{thm5}
    Let $S$ be a multiplicative subset of a ring $R$ and $M$ an $R$-module. Then the following statements hold.
    \begin{enumerate}
        \item[(1)] Let $A$ be a submodule of $B$. If $M$ is $u$-$S$-injective relative to $B$, then $M$ is $u$-$S$-injective relative to both $A$ and $\frac{B}{A}$.
        \item[(2)] $M$ is $u$-$S$-injective relative to $\bigoplus\limits_{i=1}^{n}A_i$ if and only if $M$ is $u$-$S$-injective relative to each $A_i$.
    \end{enumerate}
\end{thm}

\begin{proof}
    Similar to the proof of Theorem \ref{thm4}.
\end{proof}

\begin{corollary}\label{cor1}
    Let $S$ be a multiplicative subset of a ring $R$ and let $0\to A\xrightarrow{f} B\xrightarrow{g} C\to 0$ be a $u$-$S$-exact sequence. If $M$ is $u$-$S$-projective ($u$-$S$-injective) relative to $B$, then $M$ is $u$-$S$-projective ($u$-$S$-injective) relative to both $A$ and $C$.
\end{corollary}

\begin{proof}
   Let $0\to A\xrightarrow{f} B\xrightarrow{g} C\to 0$ be a $u$-$S$-exact sequence. Then $A$ is $u$-$S$-isomorphic to $\text{Im}(f)$ and $C$ is $u$-$S$-isomorphic to $\frac{B}{\text{Ker}(g)}$. Since $M$ is $u$-$S$-projective ($u$-$S$-injective) relative to $B$, then by Theorems \ref{thm4} and \ref{thm5}, $M$ is $u$-$S$-projective ($u$-$S$-injective) relative to both $\text{Im}(f)$ and $\frac{B}{\text{Ker}(g)}$. Hence, by Proposition \ref{pr1} (2), $M$ is $u$-$S$-projective ($u$-$S$-injective) relative to both $A$ and $C$.
\end{proof}

\begin{proposition}\label{pr2}
 Let $S$ be a multiplicative subset of a ring $R$ and let $0\to A\to B\to C\to 0$ be a $u$-$S$-split exact sequence. Then $M$ is $u$-$S$-projective ($u$-$S$-injective) relative to $B$ if and only if $M$ is $u$-$S$-projective ($u$-$S$-injective) relative to both $A$ and $C$.   
\end{proposition}

\begin{proof}
    Since the exact sequence $0\to A\to B\to C\to 0$ $u$-$S$-splits, then by \cite[Lemma 2.8]{KMOZ}, $B$ is $u$-$S$-isomorphic to $A \oplus C$. So if $M$ is $u$-$S$-projective ($u$-$S$-injective) relative to both $A$ and $C$, then by Theorems \ref{thm4} and \ref{thm5}, $M$ is $u$-$S$-projective ($u$-$S$-injective) relative to $A\oplus C$ and hence by Proposition \ref{pr1} (2), $M$ is $u$-$S$-projective ($u$-$S$-injective) relative to $B$. The other direction follows from Corollary \ref{cor1}.
\end{proof}

\begin{lemma}\label{lem2}
    Let $S$ be a multiplicative subset of a ring $R$. A $u$-$S$-exact sequence $0\to A\xrightarrow{f} B\xrightarrow{g} C\to 0$ $u$-$S$-splits if either $A$ is $u$-$S$-injective relative to $B$ or $C$ is $u$-$S$-projective relative to $B$.
\end{lemma}

\begin{proof}
   Let $0\to A\xrightarrow{f} B\xrightarrow{g} C\to 0$ be a $u$-$S$-exact sequence. Suppose that $A$ is $u$-$S$-injective relative to $B$. Then the map $$f^*: \text{Hom}_R(B,A)\to \text{Hom}_R(A,A)$$ is a $u$-$S$-epimorphism. So $s\text{Hom}_R(A,A)\subseteq \text{Im}(f^*)$ for some $s\in S$. So $s1_A\in \text{Im}(f^*)$ and hence $s1_A=f^*(f')=f'f$ for some $f'\in \text{Hom}_R(B,A)$. Thus $0\to A\xrightarrow{f} B\xrightarrow{g} C\to 0$ $u$-$S$-splits. Next, suppose that $C$ is $u$-$S$-projective relative to $B$. Then the map $$g_*: \text{Hom}_R(C,B)\to \text{Hom}_R(C,C)$$ is a $u$-$S$-epimorphism. So $t\text{Hom}_R(C,C)\subseteq \text{Im}(g_*)$ for some $t\in S$. So $t1_C\in \text{Im}(g_*)$ and hence $t1_C=g_*(g')=gg'$ for some $g\in \text{Hom}_R(C,B)$. Thus by \cite[Lemma 2.4]{ZQ}, $0\to A\xrightarrow{f} B\xrightarrow{g} C\to 0$ $u$-$S$-splits.
\end{proof}

The next two results give several characterizations of $u$-$S$-semisimple modules.
\begin{thm}\label{thm6}
  Let $R$ be a ring and $S$ a multiplicative subset of $R$. The following statements about an $R$-module $M$ are equivalent:
\begin{enumerate}
    \item[(1)] $M$ is $u$-$S$-semisimple; 
    \item[(2)] Every $R$-module is $u$-$S$-injective relative to $M$;
    \item[(3)] Every $R$-module is $u$-$S$-projective relative to $M$.
\end{enumerate}
\end{thm}
\begin{proof}
    $(1)\Rightarrow (2)$: Let $N$ be any $R$-module and $f:K\to M$ be a $u$-$S$-monomorphism. Then by (1), $0\to K\xrightarrow{f} M\to \text{Coker}(f)\to 0$ $u$-$S$-splits. So there is an $R$-homomorphism $f':M\to K$ and $s\in S$ such that $f'f=s 1_K$. For any $h\in \text{Hom}_R(K,N)$, $sh=hs1_{K}=hf'f=f^{*}(hf')$ and $hf'\in \text{Hom}_R(M,N)$. So $f^*:\text{Hom}_R(M,N)\to \text{Hom}_R(K,N)$ is a $u$-$S$-epimorphism. Thus $N$ is $u$-$S$-injective relative to $M$. Since $N$ was an arbitrary $R$-module, (2) holds. \\
 $(1)\Rightarrow (3)$: Let $N$ be any $R$-module and $g:M\to L$ be a $u$-$S$-epimorphism. Then by (1), $0\to \text{Ker}(g) \to M\xrightarrow{g} L\to 0$ $u$-$S$-splits. So there is an $R$-homomorphism $g':L\to M$ and $s\in S$ such that $gg'=s 1_L$. For any $h\in \text{Hom}_R(N,L)$, $sh=s1_{L}h=gg'h=g_{*}(g'h)$ and $g'h\in \text{Hom}_R(N,M)$. So $g_*:\text{Hom}_R(N,M)\to \text{Hom}_R(N,L)$ is a $u$-$S$-epimorphism. Thus $N$ is $u$-$S$-projective relative to $M$. Since $N$ was an arbitrary $R$-module, (3) holds. \\
$(2)\Rightarrow (1)$: Let $0\to A\to M \to C\to 0$ be a $u$-$S$-exact sequence. Then by (2), $A$ is $u$-$S$-injective relative to $M$. So by Lemma \ref{lem2}, $0\to A\to M \to C\to 0$ $u$-$S$-splits. Hence $M$ is $u$-$S$-semisimple.\\
$(3)\Rightarrow (1)$: Let $0\to A\to M \to C\to 0$ be a $u$-$S$-exact sequence. Then by (3), $C$ is $u$-$S$-projective relative to $M$. Hence by Lemma \ref{lem2}, $0\to A\to M \to C\to 0$ $u$-$S$-splits. Thus $M$ is $u$-$S$-semisimple.
\end{proof}

For an $R$-module $M$, let $E(M)$ denote the injective envelope of $M$.
\begin{thm}\label{thrm}
  Let $R$ be a ring and $S$ a multiplicative subset of $R$. The following statements about an $R$-module $M$ are equivalent:
\begin{enumerate}
    \item[(1)] $M$ is $u$-$S$-semisimple; 
    \item[(2)] For every injective $R$-module $U$ and $f\in \text{End}_{R}(U)$, $\text{Ker}(f)$ is $u$-$S$-injective relative to $M$;
    \item[(3)] For every projective $R$-module $U$ and $f\in \text{End}_{R}(U)$, $\text{Coker}(f)$ is $u$-$S$-projective relative to $M$.
\end{enumerate}
\end{thm}

\begin{proof}
$(1)\Rightarrow (2)$ and  $(1)\Rightarrow (3)$ follow from Theorem \ref{thm6}. \\
$(2)\Rightarrow (1)$: Let $N$ be any $R$-module. We show that $N$ is $u$-$S$-injective relative to $M$. Consider the injective $R$-module $U=E(N)\oplus E\big(\frac{E(N)}{N}\big)$. Define $f:U\to U$ by $f(x,y)=(0,x+N)$. Then it is easy to check that $f\in \text{End}_{R}(U)$ and $\text{Ker}(f)=N\oplus E\big(\frac{E(N)}{N}\big)$. By (2), $\text{Ker}(f)$ is $u$-$S$-injective relative to $M$. So by Theorem \ref{thm3} (2), $N$ is $u$-$S$-injective relative to $M$. Hence, every $R$-module is $u$-$S$-injective relative to $M$ and thus by Theorem \ref{thm6}, $M$ is $u$-$S$-semisimple. \\
$(3)\Rightarrow (1)$: Let $0\to A\xrightarrow{\alpha} M \xrightarrow{\beta} C\to 0$ be a $u$-$S$-exact sequence and $K=\text{Ker}(\beta)$. Then $C$ is $u$-$S$-isomorphic to $\frac{M}{K}$. Let $g:P\to M$ be an epimorphism with $P$ projective. Let $N=\text{Ker}(\eta_Kg)$, where $\eta_K:M\to \frac{M}{K}$ is the natural map. Then $N\leq P$ and $\frac{M}{K}\cong \frac{P}{N}$. There is an epimorphism $h:P'\to N$ with $P'$ projective. Let $U=P\oplus P'$, then $U$ is projective. Now define $f:U\to U$ by $f(x,y)=(h(y),0)$. Then $f\in \text{End}_{R}(U)$ and $\text{Im}(f)=N\oplus 0$. Now $\text{Coker}(f)=\frac{U}{\text{Im}(f)}=\frac{P\oplus P'}{N\oplus 0}\cong \frac{P}{N}\oplus P'$. By (3), $\text{Coker}(f)=\frac{P}{N}\oplus P'$ is $u$-$S$-projective relative to $M$. So by Theorem \ref{thm3} (1) and Proposition \ref{pr1} (1), $\frac{M}{K}\cong\frac{P}{N}$ is $u$-$S$-projective relative to $M$. But $C$ is $u$-$S$-isomorphic to $\frac{M}{K}$, so again by Proposition \ref{pr1} (1), $C$ is $u$-$S$-projective relative to $M$. Hence, by Lemma \ref{lem2}, $0\to A\xrightarrow{\alpha} M \xrightarrow{\beta} C\to 0$ is $u$-$S$-split. Therefore, $M$ is $u$-$S$-semisimple. 
\end{proof}

For a ring $R$, let $\text{Spec}(R)$ denote the set of all prime ideals of $R$ and $\text{Max}(R)$ denote the set of all maximal ideals of $R$.

\begin{lemma}\label{lemloc}
    Let $R$ be a ring and $T$ an $R$-module. Then the following statements are equivalent:
\begin{enumerate}
\item[(1)] $T=0$;
\item[(2)] $T$ is $u$-$(R\setminus \mathfrak{p})$-torsion for any $\mathfrak{p}\in \text{Spec}(R)$;
\item[(3)] $T$ is $u$-$(R\setminus \mathfrak{m})$-torsion for any $\mathfrak{m}\in \text{Max}(R)$.
\end{enumerate}
\end{lemma}
\begin{proof}
$(1)\Rightarrow (2)\Rightarrow (3)$ are clear.\\
$(3)\Rightarrow (1)$: By (3), for any $\mathfrak{m}\in \text{Max}(R)$, there exists $s_{\mathfrak{m}}\in R\setminus\mathfrak{m}$ such that $s_{\mathfrak{m}}T=0$. But since $\langle \{s_{\mathfrak{m}}\mid \mathfrak{m}\in \text{Max}(R)\}\rangle=R$, then there exist $r_1,\cdots,r_n\in R$ and $\mathfrak{m}_1,\cdots, \mathfrak{m}_n\in \text{Max}(R)$ such that $1=r_1s_{\mathfrak{m}_1}+\cdots+r_ns_{\mathfrak{m}_n}$. Hence, we have $T=(r_1s_{\mathfrak{m}_1}+\cdots+r_ns_{\mathfrak{m}_n})T\subseteq r_1s_{\mathfrak{m}_1}T+\cdots+r_ns_{\mathfrak{m}_n}T=0$. Thus $T=0$.   
\end{proof}

Let $\mathfrak{p}$ be a prime ideal of a ring $R$. We say that an $R$-module $U$ is $u$-$\mathfrak{p}$-projective ($u$-$\mathfrak{p}$-injective) relative to a module $M$ if $U$ is $u$-$(R\setminus \mathfrak{p})$-projective ($u$-$(R\setminus \mathfrak{p})$-injective) relative to $M$.  
    
\begin{proposition}\label{localr}
Let $R$ be a ring and $U,M$ be $R$-modules. Then the following statements are equivalent:
\begin{enumerate}
\item[(1)] $U$ is projective (injective) relative to $M$;
\item[(2)] $U$ is $u$-$\mathfrak{p}$-projective ($u$-$\mathfrak{p}$-injective) relative to $M$ for any $\mathfrak{p}\in \text{Spec}(R)$;
\item[(3)] $U$ is $u$-$\mathfrak{m}$-projective ($u$-$\mathfrak{m}$-injective) relative to $M$ for any $\mathfrak{m}\in \text{Max}(R)$.
\end{enumerate}
\end{proposition}

\begin{proof}
We will prove only the case of relative projectivity. The implications $(1)\Rightarrow (2)\Rightarrow (3)$ are clear.\\
$(3)\Rightarrow (1)$: Let $f:M\to N$ be any epimorphism. We show that the map $$f_*: \text{Hom}_R(U,M)\to \text{Hom}_R(U,N)$$ is an epimorphism. That is, $\text{Coker}(f_*)=0$. By $(3)$ and Theorem \ref{thm1}, $\text{Coker}(f_*)$ is $u$-$(R\setminus m)$-torsion for any $\mathfrak{m}\in \text{Max}(R)$. Thus by Lemma \ref{lemloc}, $\text{Coker}(f_*)=0$. Therefore, $U$ is projective relative to $M$. 
\end{proof}

\section{ $u$-$S$-quasi-projective and $u$-$S$-quasi-injective modules}
We begin this section with the following definition:

\begin{definition}\label{def2}
    Let $R$ be a ring and $S$ a multiplicative subset of $R$. An $R$-module $M$ is said to be $u$-$S$-quasi-projective ($u$-$S$-quasi-injective) if $M$ is $u$-$S$-projective ($u$-$S$-injective) relative to $M$.
\end{definition}

\begin{remark}\label{rem2}
 Let $R$ be a ring and $S$ a multiplicative subset of $R$.
\begin{enumerate}
    \item[(1)] By Remark \ref{rem1}, every $u$-$S$-projective ($u$-$S$-injective) $R$-module is $u$-$S$-quasi-projective ($u$-$S$-quasi-injective).
    \item[(2)] By Corollary \ref{coro}, every quasi-projective (quasi-injective) $R$-module is $u$-$S$-quasi-projective ($u$-$S$-quasi-injective).
\end{enumerate}
    
\end{remark} 

 The following example gives a $u$-$S$-quasi-projective module that is not $u$-$S$-projective.

\begin{example} \label{ex1}
    Let $R=\mathbb{Z}$ and $S=\{1,2,3,\cdots\}$. Let $\mathbb{P}$ be the set of all prime numbers. Then $M:=\bigoplus\limits_{p\in\mathbb{P}}\mathbb{Z}_p$ is a quasi-projective $R$-module \cite{AF} and hence it is $u$-$S$-quasi-projective by Remark \ref{rem2} (2). But $\text{Ext}_{R}^{1}(M,\mathbb{Z})\cong \prod\limits_{p\in\mathbb{P}}\text{Ext}_{R}^{1}(\mathbb{Z}_p,\mathbb{Z})\cong \prod\limits_{p\in\mathbb{P}}\mathbb{Z}_p$ is not $u$-$S$-torsion, so $M$ is not $u$-$S$-projective by \cite[Proposition 2.9]{ZQ}.
\end{example}
   
\begin{proposition}\label{local}
      Let $R$ be a ring and $M$ an $R$-module. Then the following statements are equivalent:
      \begin{enumerate}
          \item[(1)] $M$ is quasi-projective (quasi-injective);
          \item[(2)] $M$ is $u$-$\mathfrak{p}$-quasi-projective ($u$-$\mathfrak{p}$-quasi-injective) for any $\mathfrak{p}\in \text{Spec}(R)$;
          \item[(3)] $M$ is $u$-$\mathfrak{m}$-quasi-projective ($u$-$\mathfrak{m}$-quasi-injective) for any $\mathfrak{m}\in \text{Max}(R)$.
      \end{enumerate}
\end{proposition}

\begin{proof}
This follows from Proposition \ref{localr}. 
\end{proof}

 The following example provides a $u$-$S$-quasi-injective module that is not $u$-$S$-injective.
 
\begin{example}\label{ex2}
    Let $R=\mathbb{Z}$ and $M=\mathbb{Z}_2$. Then $M$ is a simple $R$-module and hence it is quasi-injective. But $M$ is not injective $R$-module, so by \cite[proposition 4.8]{QK}, $M$ is not $u$-$\mathfrak{m}$-injective for some maximal ideal $\mathfrak{m}$ of $R$. However, $M$ is $u$-$\mathfrak{m}$-quasi-injective by Remark \ref{rem2} (2). 
\end{example}

\begin{proposition}\label{psemi}
 Let $S$ be a multiplicative subset of a ring $R$ and $M$ a $u$-$S$-semisimple $R$-module. Then $M$ is both $u$-$S$-quasi-injective and $u$-$S$-quasi-projective.
\end{proposition}
    
\begin{proof}
    Let $M$ be a $u$-$S$-semisimple module. Then by Theorem \ref{thm6}, every $R$-module is $u$-$S$-injective ($u$-$S$-projective) relative to $M$. In particular, $M$ is $u$-$S$-injective ($u$-$S$-projective) relative to $M$. Thus $M$ is both $u$-$S$-quasi-injective and $u$-$S$-quasi-projective.
\end{proof}

The following example gives a $u$-$S$-quasi-injective module that is not quasi-injective.

\begin{example} \label{ex3}
    Let $R=\mathbb{Z}$ and $S=\mathbb{Z}\setminus\{0\}$. By \cite[Example 3.7]{ZQ}, $\mathbb{Z}$ is a $u$-$S$-semisimple $\mathbb{Z}$-module. So by Proposition \ref{psemi}, $\mathbb{Z}$ is a $u$-$S$-quasi-injective $\mathbb{Z}$-module. However, $\mathbb{Z}$ is not a quasi-injective $\mathbb{Z}$-module by \cite[Example 2.3]{KG}.
\end{example}

\begin{proposition}\label{prop1}
  Let $S$ be a multiplicative subset of a ring $R$. Then $\bigoplus\limits_{i=1}^{n}A_i$ is $u$-$S$-quasi-projective ($u$-$S$-quasi-injective) if and only if $A_i$ is $u$-$S$-projective ($u$-$S$-injective) relative to $A_j$ for each $i,j=1,2,\cdots,n$.  
\end{proposition}

\begin{proof}
    By Theorems \ref{thm3}, \ref{thm4}, and \ref{thm5}, we have $\bigoplus\limits_{i=1}^{n}A_i$ is $u$-$S$-quasi-projective ($u$-$S$-quasi-injective) if and only if $\bigoplus\limits_{i=1}^{n}A_i$ is $u$-$S$-projective ($u$-$S$-injective) relative to $\bigoplus\limits_{i=1}^{n}A_i$ if and only if $A_i$ is $u$-$S$-projective ($u$-$S$-injective) relative to $\bigoplus\limits_{j=1}^{n}A_j$ for each $i=1,2,\cdots,n$ if and only if $A_i$ is $u$-$S$-projective ($u$-$S$-injective) relative to $A_j$ for each $i,j=1,2,\cdots,n$.    
\end{proof}

\begin{proposition}\label{propq2}
     Let $S$ be a multiplicative subset of a ring $R$ and $B$ a $u$-$S$-quasi-projective $R$-module. Then an exact sequence $0\to A\xrightarrow{f} B\xrightarrow{g} C\to 0$ is $u$-$S$-split if and only if $C$ is $u$-$S$-projective relative to $B$.
\end{proposition}

\begin{proof}
 Let $B$ be a $u$-$S$-quasi-projective module and let $0\to A\xrightarrow{f} B\xrightarrow{g} C\to 0$ be a $u$-$S$-split exact sequence. Then by \cite[Lemma 2.8]{KMOZ}, $B$ is $u$-$S$-isomorphic to $A \oplus C$. So by Proposition \ref{pr1}, $A \oplus C$ is $u$-$S$-quasi-projective. Hence by Proposition \ref{prop1}, $C$ is $u$-$S$-projective relative to both $A$ and $C$. Thus by Proposition \ref{pr2}, $C$ is $u$-$S$-projective relative to $B$. The converse follows from Lemma \ref{lem2}.
\end{proof}

\begin{proposition}
     Let $S$ be a multiplicative subset of a ring $R$ and $B$ a $u$-$S$-quasi-injective $R$-module. Then an exact sequence $0\to A\xrightarrow{f} B\xrightarrow{g} C\to 0$ is $u$-$S$-split if and only if $A$ is $u$-$S$-injective relative to $B$.
\end{proposition}

\begin{proof}
    The proof is similar to that of Proposition \ref{propq2}
\end{proof}

Lastly, we give a characterization of $u$-$S$-semisimple rings in terms of $u$-$S$-quasi-injective and $u$-$S$-quasi-projective modules.

\begin{thm}\label{thm7}
  Let $R$ be a ring and $S$ a multiplicative subset of $R$. Then the following statements are equivalent:
\begin{enumerate}
    \item[(1)] $R$ is a $u$-$S$-semisimple ring; 
    \item[(2)] Every $R$-module is $u$-$S$-quasi-injective;
    \item[(3)] Every $R$-module is $u$-$S$-quasi-projective.
\end{enumerate}
\end{thm}

\begin{proof}
    $(1)\Rightarrow (2)$ and $(1)\Rightarrow (3)$ follow from \cite[Theorem 3.5]{ZQ} and Remark \ref{rem2} (1).\\
 $(2)\Rightarrow (1)$: Let $M$ be any $R$-module. Then for any $R$-module $N$, $M\oplus N$ is $u$-$S$-quasi-injective by (2). So for any $R$-module $N$, $N$ is $u$-$S$-injective relative to $M$ by Proposition \ref{prop1}. Thus, every $R$-module is $u$-$S$-injective relative to $M$. Hence, by Theorem \ref{thm6}, $M$ is $u$-$S$-semisimple. Since $M$ was an arbitrary $R$-module, every $R$-module is $u$-$S$-semisimple. Thus, by \cite[Theorem 3.5]{ZQ}, $R$ is a $u$-$S$-semisimple ring. \\ 
  $(3)\Rightarrow (1)$: Similar to the proof of the implication $(2)\Rightarrow (1)$.  
\end{proof}

\end{document}